\documentclass[]{article}

\pdfoutput=1
\usepackage{amsmath}
\usepackage{amssymb}
\usepackage{graphicx}

\usepackage{bm}
\usepackage{natbib}
\usepackage{setspace}
\usepackage{xcolor}

\usepackage[plain,noend]{algorithm2e}

\usepackage{natbib}
\usepackage[colorlinks = true,
            linkcolor = blue,
            urlcolor  = blue,
            citecolor = blue,
            anchorcolor = blue]{hyperref}
\usepackage[margin=1in]{geometry}
\usepackage{setspace}
\usepackage{graphicx}
\usepackage[
singlelinecheck=false,
labelfont=bf,font={doublespacing, small}
]{caption}

\usepackage[affil-it]{authblk} 
\usepackage{etoolbox}
\usepackage{lmodern}

\def\var{{\mathrm{var}}}

\newcommand\norm[1]{\lVert#1\rVert}

\makeatletter
\patchcmd{\@maketitle}{\LARGE \@title}{\fontsize{16}{19.2}\selectfont\@title}{}{}
\makeatother

\newtheorem{theorem}{Theorem}
\newtheorem{lemma}{Lemma}

\usepackage{sectsty}
\allsectionsfont{\sffamily}

\begin{document}

\title{Improved convergence rates of nonparametric penalized regression under misspecified total variation}

\author[1]{Marlena Bannick}
\author[1]{Noah Simon}

\affil[1]{Department of Biostatistics, University of Washington, Seattle, Washington, U.S.A.}


\maketitle

\begin{abstract}
    Penalties that induce smoothness are common in nonparametric regression. In many settings, the amount of smoothness in the data generating function will not be known. \citet{simonConvergenceRatesNonparametric2021} derived convergence rates for nonparametric estimators under misspecified smoothness. We show that their theoretical convergence rates can be improved by working with convenient approximating functions. Properties of convolutions and higher-order kernels allow these approximation functions to match the true functions more closely than those used in \citet{simonConvergenceRatesNonparametric2021}. As a result, we obtain tighter convergence rates.
\end{abstract}

\newpage
\section{Introduction}

Suppose we are interested in estimating a regression function $f^*$, with 
\begin{align}\label{eq:f-star}
	y_i = f^*(x_i) + \epsilon_i
\end{align}
where $\epsilon_i$ are independent, with $E(\epsilon_i|x_i) = 0$, $\var(\epsilon|x_i) = \sigma^2$, and $x_i \in [0, 1]$, with $x_i$ equally spaced. In many cases, little is known about $f^*$ beyond perhaps general smoothness and structural constraints. In that case, nonparametric techniques are used. One prominent nonparametric technique is penalized regression, which is given as the solution to a functional optimization problem. In this case a penalty function, $P$, is selected based on known structure and smoothness in $f^*$, and one obtains an estimator by solving:
\begin{align}\label{eq:f-hat}
	\hat{f} = \arg\min_{f\in\mathcal{F}} \frac{1}{2}||y - f||_n^2 + \lambda_n P(f)
\end{align}
where $||y - f||_n^2 = n^{-1} \sum_{i=1}^{n} (y_i - f(x_i))^2$ denotes the empirical norm, and $\lambda_n>0$ is a tuning parameter that balances MSE and roughness, selected as a function of the sample size, $n$. 

A commonly used penalty function $P$ is the $k$th-order total variation of $f$:
\begin{align}\label{total-variation}
P_k(f) = \sup_{x_1, ..., x_M} \sum_{m=1}^{M} \left|f^{(k-1)}(x_{m+1}) - f^{(k-1)}(x_{m})\right|
\end{align}
where $f^{(k)}$ is the $k$th derivative of $f$ with associated function class $\mathcal{F}_k$. $\mathcal{F}_k$ is the class of functions with finite $k$th-order total variation $\mathcal{F}_{k} = \{f: [-1, 1] \to \mathbb{R}; f \text{ is $k-1$ times weakly differentiable, } P_k(f) < \infty\}$ and with $\mathcal{F}_k$ a linear subspace of $\mathcal{L}_2[-1,1]$.

There are many contemporary extensions/modifications that make use of this type of penalty. These include the Lasso Isotone \citep{fangLASSOIsotoneHighDimensional2012}, the fused Lasso for additive models \citep{tibshiraniSparsitySmoothnessFused2005, petersenFusedLassoAdditive2016}, additive regression splines \citep{jhongAdditiveRegressionSplines2022}, additive models for quantile regression \citep{koenkerAdditiveModelsQuantile2011}, and generalized sparse additive models \citep{harisGeneralizedSparseAdditive}. Recent work has focused on efficient computation of these estimates \citep{sadhanalaAdditiveModelsTrend2019, sadhanalaMultivariateTrendFiltering2021, tibshiraniDividedDifferencesFalling2022, wangFallingFactorialBasis2014a, tibshiraniAdaptivePiecewisePolynomial2014}.

It is of general interest to understand the rate of convergence of $\hat{f}$, obtained by solving \eqref{eq:f-hat} with penalty $P_k$, to $f^{*}$ (and eg. particularly critical when using $\hat{f}$ as a nuisance estimate when engaging with a pathwise differentiable parameter). If $f^*$ has the requisite smoothness (i.e., $f^{*}\in \mathcal{F}_k$), then standard results show that $\hat{f}$ achieves the minimax rate over an appropriately bounded subset of $\mathcal{F}_k$. However, if the true data-generating function is not smooth enough ($f^{*}\not\in \mathcal{F}_k$), these convergence rates do not hold. \cite{simonConvergenceRatesNonparametric2021} derive convergence rates for $\hat{f}$ in \eqref{eq:f-hat} with misspecified smoothness (when $f^{*}\in \mathcal{F}_l$ for some $\ell < k$). Their results indicate that convergence rates depend on both the assumed ($k$) and true ($l$) levels of smoothness. Based on their empirical results, \citet{simonConvergenceRatesNonparametric2021} hypothesize that their convergence rates are not sharp in many settings. In this manuscript we give sharper convergence rates that match the empirical results in \cite{simonConvergenceRatesNonparametric2021}.

We preview our results in Table \ref{tab-1}. We emphasize that our results apply only when $f^*$ does not have the requisite smoothness ($\ell < k$). When $\ell = 1$ (e.g., $f^*$ is a piecewise constant function), our rates match those of \citet{simonConvergenceRatesNonparametric2021}. When $\ell > 1$, we improve on the rates in \citet{simonConvergenceRatesNonparametric2021}. Based on empirical results, our rates in both scenarios appear to be sharp. This closes a gap in the literature for convergence rates of under misspecified smoothness.

\begin{table}[htbp]
\centering
\caption{A comparison of convergence rates derived in \citet{simonConvergenceRatesNonparametric2021}, and our rates, when a $k$th order total variation penalty is used to estimate a function $f^* \in \mathcal{F}_{\ell}$. $^{1}$See van de Geer (2000) Theorem 10.2 and Section 10.1.2}\label{tab-1}
\begin{tabular}{c|c|cc}
\hline
& \multicolumn{1}{c}{Correct Specification} & \multicolumn{2}{|c}{Misspecification} \\
\hline
  & $k \leq \ell$  & $k > \ell$, $\ell = 1$ & $k > \ell$, $\ell > 1$ \\
  \hline\hline
  van de Geer (2000)$^{1}$ & $n^{-2k/(2k+1)}$ & -- & -- \\
  \citet{simonConvergenceRatesNonparametric2021} & -- & $n^{-2k/(4k-1)}$ & $n^{-2k/(3k-\ell + 1)}$ \\
  Theorem \ref{main-thm-main} & -- & $n^{-2k/(4k-1)}$ & $n^{-2k(2\ell-1)/(4k\ell-1)}$ \\
  \hline
\end{tabular}
\end{table}

\section{Preliminaries}

Here we summarize the strategy that \citet{simonConvergenceRatesNonparametric2021} use to derive convergence rates (Theorem \ref{theorem-1}). The key element of their strategy is to find a sequence of functions $f_1, ..., f_n$ that converge to $f^*$. But unlike $f^*$, these functions each live in $\mathcal{F}_k$. We will use the same strategy to derive our convergence rates, but with a different sequence of approximation functions.

In Theorem \ref{theorem-1}, we restate \citet{simonConvergenceRatesNonparametric2021}'s result that the MSE of $\hat{f}$ can be upper bounded by two terms related to the approximation sequence. The first term is the approximation error. This is the MSE between the true function $f^*$ and the sequence $f_n$. As $n \to \infty$, the approximation error should vanish. The second term is related to the smoothness of the approximating sequence. Since  $f^*$ lives in the closure of $\mathcal{F}_k$ and not $\mathcal{F}_k$ itself, the $k$th order total variation of $f_n$ should increase without bound as $f_n$ approaches $f^*$. To upper bound the MSE, the approximation error and smoothness terms must be carefully balanced. This can be accomplished by choosing $f_1, ..., f_n$ with care.
\begin{theorem}[Theorem 2.2 of \citet{simonConvergenceRatesNonparametric2021}]\label{theorem-1}
	Suppose data are generated according to \eqref{eq:f-star}, $\hat{f}$ is defined as in \eqref{eq:f-hat} for some $\lambda_n > 0$, $P_k$ is defined as in \eqref{total-variation}, and $$\mathcal{F}_{k} = \{f: [-1, 1] \to \mathbb{R}; f \text{ is $k$ times weakly differentiable, } P_k(f) < \infty\}$$ is a linear subspace of $\mathcal{L}_2[-1,1]$. Let $f_0, ..., f_n \in \mathcal{F}_k$ with $P_k(f_j) > 0$ for all $j$.
	If we choose
 \begin{align*}
     \lambda_n = O_p\left\{n^{-2/(2+k^{-1})}P_k^{(k^{-1}-2)/(k^{-1}+2)}(f_n)\right\},
 \end{align*}
 then
\begin{align}\label{inequality}
||\hat{f} - f^*||_n^2 \leq ||f^* - f_n||_n^2 + O_p\left\{\lambda_n P_k(f_n)\right\}.
\end{align}
\end{theorem}

For $f^* \in \mathcal{F}_{\ell}$, \citet{simonConvergenceRatesNonparametric2021} chose integrals of b-splines as their approximating sequences. These smooth over a small domain around each discontinuity in the $\ell$th derivative of $f^*$. When $\ell > 1$, these approximation functions are not ideal: Consider $f^{*}$ that is infinitely smooth outside of a finite set of discontinuities in higher order derivatives. It is ideal to identify approximations $f_1, ..., f_n$ that exactly agree with $f^{*}$ except in vanishing neighborhoods around those discontinuities. That is not the case for those $f_n$ proposed in \citet{simonConvergenceRatesNonparametric2021}.
This leads to excessive approximation error. For example, when a 3rd-order penalty is used to estimate $f^{*}$ that is a piecewise 2nd-order polynomial, their upper bound obtained via Theorem \ref{theorem-1} does not match their empirical results (it is too loose).

\section{Improved convergence rates using higher-order kernels}

We specify a different sequence of approximation functions that allows us to get a tighter bound on the approximation error. We denote our approximation functions as $f_{\delta,k}$, parameterized by $\delta = \delta(n)$. To construct $f_{\delta,k}$, we convolve $f^*$ with a $k$th-order kernel function $H_{k,\delta}$ that has bounded support on $[-\delta,\delta]$, integrates to 1, and has bounded derivatives \citep{tsybakov2009}. This smoothes over discontinuities in the $\ell$th derivative of $f^*$ using small windows of size $[-\delta, \delta]$. Importantly, convolving a higher order kernel with a polynomial of a lower order leaves the polynomial unchanged (see Supplemental Material Lemma C.2). This is why $f_{\delta,k}$ matches $f^*$ outside of the small domain around each discontinuity. More specifically, we let $f_{\delta,k}$ be defined as follows:
\begin{align*}
	f_{\delta,k} = f^* \star H_{k,\delta} = x &\to \int_{-\infty}^{\infty} f^*(x-t) H_{k, \delta}(t) dt = \int_{-\delta}^{\delta} f^*(x-t) H_{k, \delta}(t) dt.
 \end{align*}
The use of convolutions to smooth out rough behavior is an established proof technique in differential equations. Convolution with higher order kernels is also prominently used, methodologically, in non-parametric regression (e.g., in kernel density estimation, and nadaraya-watson estimators). However, to the best of our knowledge, using convolution with higher order kernels to identify smooth oracle sequences is novel. Around the discontinuities, the convolution returns a kernel-weighted average of $f^*$ (some of these weights will necessarily be negative if $k \geq 2$). We show an example of $f^*$, $H_{k,\delta}$, and $f_{\delta,k}$ in Figure \ref{fig:approx-funcs}.
\begin{figure}[htbp]
    \centering
    \caption{Plot of the original function $f^* \in \mathcal{F}_{\ell}$ (black line) convoluted with a higher-order kernel $H_{k,\delta}$, resulting in $f_{\delta,k}$ (blue line). On the left, $\ell = 1, k = 2$, and on the right $\ell = 2, k = 3$. The shaded region indicates the bandwidth $\delta$, in which the approximation function $f_{\delta,k}$ may differ from the true function $f^*$.}
\includegraphics[width=\textwidth]{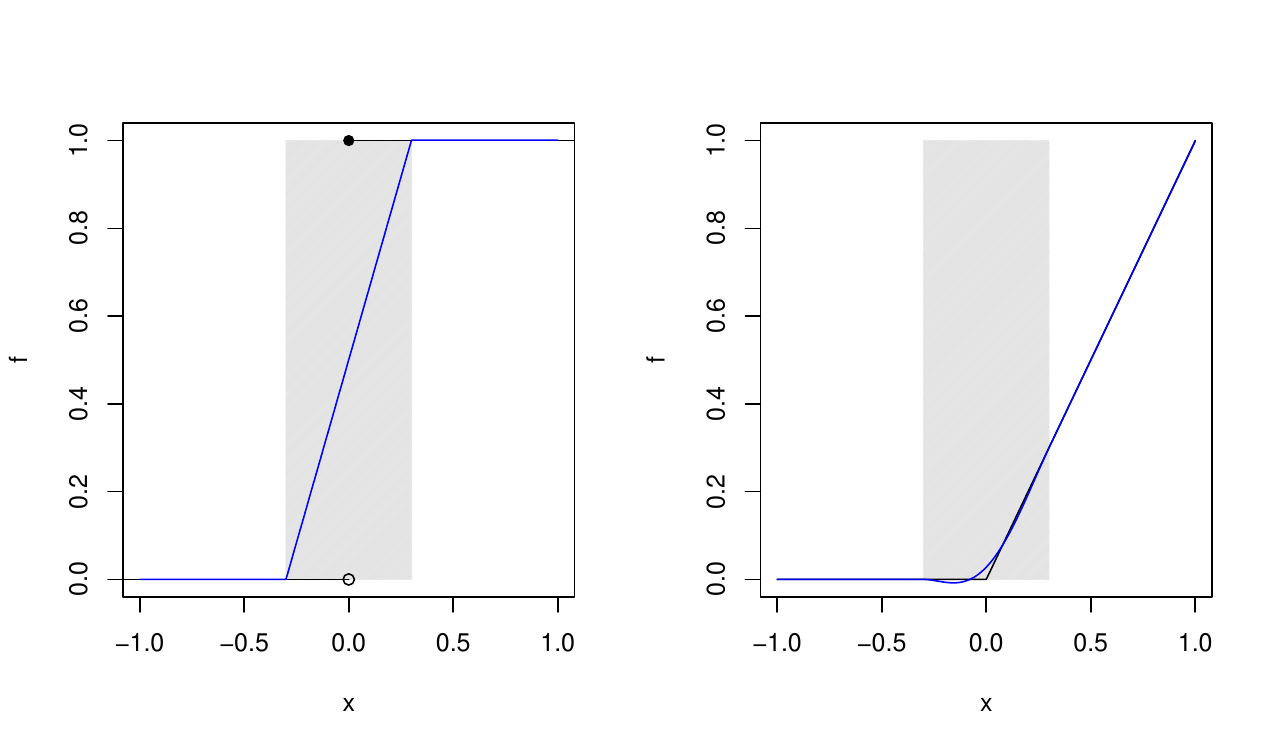}
    \label{fig:approx-funcs}
\end{figure}

Using $f_{k,\delta}$, we derive upper bounds on both the approximation error (Lemma \ref{approx-error}) and the value of penalty function (Lemma \ref{penalty-main}). The upper bounds are then used to obtain an upper bound on the MSE using \eqref{inequality}.

\begin{lemma}[Upper Bound on Approximation Error]\label{approx-error}
	Let $f^* \in \mathcal{F}_{\ell}$, and $k > \ell$. Then $||f^* - f_{\delta,k}||_n^2 = O_p(\delta^{2\ell-1})$.
\end{lemma}

\begin{lemma}[Upper Bound on Penalty Function]\label{penalty-main}
	Let $f^* \in \mathcal{F}_{\ell}$, and $k > \ell$. Then $P_k(f_{\delta,k}) = O(1/\delta^{k-\ell})$. Specifically,
	\begin{align*}
		P_k(f_{\delta,k}) \leq C P_{\ell}(f^*) / \delta^{k-\ell}
	\end{align*}
	where $C$ is a constant that depends on $k$ and $\ell$, and $ P_{\ell}(f^*)$ is the $\ell$th order total variation penalty of $f^*$.
\end{lemma}

We can now combine Lemmas \ref{approx-error} and \ref{penalty-main}, with Theorem \ref{theorem-1} to obtain rates on $\norm{\hat{f} - f^*}_n^2$. The main work of Theorem \ref{main-thm-main} is to optimize over $\delta$ as a function of $n$ to balance the approximation error and smoothness terms.
\begin{theorem}[Improved Convergence Rates]\label{main-thm-main}
	Let $f^* \in \mathcal{F}_{\ell}$. If we choose $$\lambda_n = n^{-2k(k + \ell-1)/(4k\ell - 1)} \{CP_{\ell}(f^*)\}^{(k^{-1} - 2)/(k^{-1} + 2)}$$ with $C$ from Lemma \ref{penalty-main}
 and a $k$th order total variation penalty, $P_k(\cdot)$ is used to estimate $\hat{f}$ in \eqref{eq:f-hat}, then we have
	\begin{align*}
		||f^* - \hat{f}||_n^2 = O_p\left\{n^{\frac{-2k (2\ell - 1)}{4k\ell - 1}} \right\}.
	\end{align*}
\end{theorem}
This result is obtained by selecting $\delta(n) = O\left\{n^{-2k/(4k\ell - 1)}\right\}$ for our oracle approximating sequence. That choice balances the two terms in \eqref{inequality}.

This is a faster rate than obtained in \citet{simonConvergenceRatesNonparametric2021} (see Table \ref{tab-1}); e.g., for $k = 3$ and $\ell = 2$, our rate is 0.783, whereas theirs is 0.75. Our rate matches the empirical rates observed in that paper (which we replicate in the appendix), leading us to hypothesize that our rates are sharp. However, in order to prove such a statement, we would need to derive a lower bound on the MSE under misspecified smoothness, which is beyond the scope of this article.

\bibliographystyle{apalike}
\bibliography{bib}

\end{document}


\maketitle

\tableofcontents

The goal of this appendix is that it is relatively self-contained. As such, we have re-proved results from others in detail, namely \citet{simonConvergenceRatesNonparametric2021} in Section \ref{app:simon} and \citet{birmanPIECEWISEPOLYNOMIALAPPROXIMATIONSFUNCTIONS1967} in Section \ref{app:birman}. The proofs behind our new work are included in Section \ref{app:new}. In Section \ref{app:sim}, we restate the simulation results from \citet{simonConvergenceRatesNonparametric2021} for completeness, and to show how they match up with our theoretical rates.

\appendix
\newpage
\section{Proof of Theorem 3.1 in \cite{simonConvergenceRatesNonparametric2021}}\label{app:simon}

In this appendix, we re-prove Theorem 3.1 in \cite{simonConvergenceRatesNonparametric2021}. The results in this appendix are not new, but we provide comprehensive proofs for the convenience of the reader. We start by proving Lemma \ref{lemma:s1.1} and Lemma \ref{lemma:s1.2}, which are Lemmas S1.1 and S1.2 in \citet{simonConvergenceRatesNonparametric2021}, respectively. These two results are combined to prove Lemma \ref{lemma:2.1}, which is Lemma 2.1 in \citet{simonConvergenceRatesNonparametric2021}. The purpose of Lemma \ref{lemma:2.1} is to express an upper bound for the MSE $\norm{\hat{f} - f^*}^2_n$ (where $f^* \notin \mathcal{F}_k$) in terms of some other function $f^O$ and the estimated function $\hat{f}$ (both of which live in $\mathcal{F}_k$).

\begin{lemma}\label{lemma:s1.1}
	Suppose that $\hat{f}$, $f^*$, and $f^O$ are functions that map to $\mathbb{R}$. Then
	\begin{align*}
		2\langle \hat{f} - f^*, \hat{f} - f^O \rangle_n = \norm{\hat{f} - f^*}_n^2 + \norm{\hat{f} - f^O}_n^2 - \norm{f^* - f^O}_n^2.
	\end{align*}
\end{lemma}
\begin{proof}
\begin{align*}
	\norm{f^* - f^O}_n^2 &= \norm{(f^* - \hat{f}) + (\hat{f} - f^O)}_n^2 \\
	&= \norm{f^* - \hat{f}}_n^2 + \norm{\hat{f} - f^O}_n^2 + 2\langle f^* - \hat{f}, \hat{f} - f^O \rangle_n \\
	&= \norm{f^* - \hat{f}}_n^2 + \norm{\hat{f} - f^O}_n^2 - 2 \langle\hat{f} - f^*, \hat{f} - f^O \rangle_n.
\end{align*}
\end{proof}

\begin{lemma}\label{lemma:s1.2}
	Let $\hat{f}$ be defined as in (3) and let $f^O \in \mathcal{F}$ be any other function, where $\mathcal{F}$ is a linear subspace of $\mathcal{L}_2[-1, 1]$, and $P$ is a seminorm. Then
	\begin{align*}
		\langle y - \hat{f}, f^O - \hat{f} \rangle _n \leq \lambda P(f^O) - \lambda P(\hat{f}).
	\end{align*}
\end{lemma}
\begin{proof}
	Let $\epsilon \in [0, 1]$, and define $f_{\epsilon} = \hat{f} + \epsilon(f^O - \hat{f})$, such that $\epsilon$ represents the deviation in $f_{\epsilon}$ from $\hat{f}$ towards $f^O$. Since $\mathcal{F}$ is a linear subspace, it is convex, which means that a convex combination of functions in $\mathcal{F}$, $f_{\epsilon} \equiv (1-\epsilon)\hat{f} + \epsilon f^O \in \mathcal{F}$ as well, since $\hat{f} \in \mathcal{F}$ by definition (3).
	
	The original problem that we are working with is $\hat{f} = \arg\min_{f\in \mathcal{F}} (1/2) \norm{y - f}_n^2 + \lambda P(f)$. Since $\hat{f}$ is the minimizer for this problem, we know that $\hat{\epsilon} = 0$ is the minimizer for the problem in $\epsilon \in [0, 1]$
	\begin{align*}
		(1/2) \norm{y - (1-\epsilon) \hat{f} - \epsilon f^O}_n^2 + \lambda P\{(1-\epsilon) \hat{f} + \epsilon f^O\}.
	\end{align*}
	Then, using the stationarity KKT condition, since there is an equivalence between the unconstrained problem above and a constrained problem ($P(f_{\epsilon}) \leq c$):
	\begin{align*}
		\partial \left[(1/2) \norm{y - f_{\epsilon}}_n^2 + \lambda P(f_{\epsilon})\right]\big|_{\epsilon = 0} \in 0.
	\end{align*}
	For each of the subgradients, we have
	\begin{align*}
		\partial (1/2) \norm{y - f_{\epsilon}}_n^2|_{\epsilon=0} &= -\langle y - f_{\epsilon}, (d/d\epsilon) f_{\epsilon} \rangle_n|_{\epsilon = 0} \\
		&= - \langle y - f_{\epsilon}, -\hat{f} + f^O \rangle_n|_{\epsilon = 0} \\
		&= -\langle y - \hat{f}, f^O - \hat{f} \rangle _n \\
		\partial \lambda P(f_{\epsilon})|_{\epsilon = 0} &= \lambda \langle\dot{P}(f_{\epsilon}), (d/d\epsilon) f_{\epsilon} \rangle_n|_{\epsilon = 0} \\
		&= \lambda \langle\dot{P}(\hat{f}), f^O -\hat{f} \rangle_n \\
		\text{(definition of a subgradient)} &\leq \lambda \left[P(f^O) - P(\hat{f})\right].
	\end{align*}
	Therefore, $\langle y - \hat{f}, f^O - \hat{f} \rangle _n \leq \lambda \left\{P(f^O) - P(\hat{f})\right\}$.
\end{proof}

We now combine Lemmas \ref{lemma:s1.1} and \ref{lemma:s1.2} to obtain an expression for the MSE between $f^*$ and $\hat{f}$ in terms of the MSE between $\hat{f}$ and $f^O$, the MSE between $f^O$ and $f^*$, and quantities related to $\hat{f}$ and $f^O$. Because we will eventually construct a sequence of $f^O$ that converge to $f^*$, the result of Lemma \ref{lemma:2.1} will be critical in determining our convergence rates. This is not the final result that we will want to use; in Theorem \ref{main-thm}, we will refine this bound further so that it does not depend at all on $\hat{f}$ (except for $\norm{\hat{f} - f^*}_n^2$), and can be manipulated solely using a sequence of $f^O$.

\begin{lemma}\label{lemma:2.1}
	Define $f^*$ as in (1) (potentially not in $\mathcal{F}$), and $\hat{f}$ as in (3) (in $\mathcal{F}$). Let $f^O$ be any other function in $\mathcal{F}$. Suppose that $P$ is convex. Then
	\begin{align*}
		\norm{\hat{f} - f^*}_n^2 + \norm{\hat{f} - f^O}_n^2 \leq \norm{f^O - f^*}_n^2 + 2\langle \epsilon, \hat{f} - f^O\rangle_n + 2\lambda \{P(f^O) - P(\hat{f})\}.
	\end{align*}
\end{lemma}
\begin{proof}
The proof relies on application of Lemmas \ref{lemma:s1.1} and \ref{lemma:s1.2}. Recall that $\epsilon = y - f^*$. Then:
\begin{align*}
	\text{(Lemma \ref{lemma:s1.1})} \quad  \norm{f^* - \hat{f}}_n^2 + \norm{\hat{f} - f^O}_n^2 &\leq \norm{f^* - f^O}_n^2 + 2\langle \hat{f} - f^*, \hat{f} - f^O \rangle_n \\
	&= \norm{f^* - f^O}_n^2 + 2\langle y - \hat{f}, f^O - \hat{f} \rangle_n - 2\langle \epsilon, f^O - \hat{f} \rangle_n \\
	\text{(Lemma \ref{lemma:s1.2})} \quad &\leq \norm{f^* - f^O}_n^2 + 2\lambda \{P(f^O) - P(\hat{f})\} - 2 \langle \epsilon, f^O - \hat{f}\rangle_n \\
	&= \norm{f^* - f^O}_n^2 + 2\lambda \{P(f^O) - P(\hat{f})\} + 2 \langle \epsilon, \hat{f} - f^O \rangle_n
\end{align*}
\end{proof}

The purpose of Lemma \ref{lemma:vandegeer} is to obtain a high probability upper bound on the last term of Lemma \ref{lemma:2.1}, $\langle \epsilon, \hat{f} - f^O\rangle_n$. We do this through using a concentration result from \citet{geerEmpiricalProcessesMEstimation2009} which uses the metric entropy of a normalized version of a function class $\mathcal{F}$ such that its penalty function $P$ is bounded by 1. Thus, in order to apply this result to our scenario in Theorem 1, we need to use the fact that the metric entropy of $\{f \in \mathcal{F}_k|P_k(f) \leq 1\}$ has the desired polynomial bound \citep{mammenNonparametricRegressionQualitative1991, geerEmpiricalProcessesMEstimation2009}.

\begin{lemma}\label{lemma:vandegeer}
	Suppose data are generated according to (1), and $\hat{f}$ is defined as in (3), for some $\lambda_n > 0$. Suppose $P$ is a seminorm and $\mathcal{F}$ is a linear subspace of $\mathcal{L}_2[-1, 1]$. Let $f_0^O, ..., f_n^O \in \mathcal{F}$ with $P(f^O_k) > 0$ for all $k$. Suppose the metric entropy of $\mathcal{F}$ has the following polynomial bound, with $\alpha \in (0, 2)$ and for some $A > 0$
	\begin{align*}
		H[\delta, \{f \in \mathcal{F} | P(f) \leq 1\}, ||\cdot||_n] \leq A\delta^{-\alpha}.
	\end{align*}
	Then
	\begin{align*}
		\sup_{f \in \mathcal{F}} \frac{\langle \epsilon, f - f_n^O \rangle_n}{\norm{f - f_n^O}_n^{1-\alpha/2} \{P(f) + P(f_n^O)\}^{\alpha/2}} \leq C_{\varepsilon} n^{-1/2}
	\end{align*}
	with probability at least $1 - \varepsilon$, where $C_{\varepsilon}$ only depends on $\varepsilon$.
\end{lemma}
\begin{proof}
	Assuming that $P(f) + P(f^O_n) > 0$, then we can instead work with the normalized functions $(f - f^O_n) / \{P(f) + P(f_n^O)\}$ because $f - f_n^O \in \mathcal{F}$ since $\mathcal{F}$ is a linear subspace, and because
	\begin{align*}
	P[(f - f^O_n) / \{P(f) + P(f_n^O)\}] = P(f - f^O_n)/\{P(f) + P(f_n^O)\}
	\end{align*}
	and $P(f - f^O_n) \leq \{P(f) + P(f_n^O)\}$ by triangle inequality, so $P[(f - f^O_n) / \{P(f) + P(f_n^O)\}] \leq 1$, thus the polynomial bound holds for these normalized functions.
	Let $g := (f - f^O_n) / \{P(f) + P(f_n^O)\} \in \mathcal{G}$, which satisfies the same polynomial bound. Lemma 8.4 of \citet{geerEmpiricalProcessesMEstimation2009} tells us that as long as the polynomial bound given in the statement of the lemma is satisfied, that $\sup_{g \in \mathcal{G}} \norm{g}_n \leq R$, and
	\begin{align*}
		\max_{i = 1, ..., n} K^2 \E\{\exp(|\epsilon_i|^2/K^2) - 1\} \leq \sigma^2_0
	\end{align*}
	for some $K$ and $\sigma_0^2$,
	then for some constant $c$ depending on $A$, $\alpha$, $R$, $K$, and $\sigma_0$, for all $T \geq c$,
	\begin{align*}
		P \left\{\sup_{g \in \mathcal{G}} \frac{\left|\frac{1}{\sqrt{n}} \sum_{i=1}^{n}  \epsilon_i g(x_i) \right|}{\norm{g}_{n}^{1-\alpha/2}}  \geq T \right\} \leq c \exp(-T^2 / c^2).
	\end{align*}
	(setting $\epsilon_i = W_i$ in the notation of \citet{geerEmpiricalProcessesMEstimation2009}). This is a condition on sub-Gaussianity; since we are working with Gaussian errors $\epsilon_i$, this condition is satisfied. We can flip this probability, note that the numerator in the supremum is $\sqrt{n} \langle \epsilon, g\rangle_n$, and set $\varepsilon = c \exp(-T^2 / c^2)$, to get
	\begin{align*}
		P \left(\sup_{g \in \mathcal{G}} \frac{\left|\langle \epsilon, g\rangle_n\right|}{\norm{g}_{n}^{1-\alpha/2}}  \leq n^{-1/2} C_{\varepsilon} \right) \geq 1 - \varepsilon
	\end{align*}
	where $C_{\varepsilon} = c\sqrt{\log(c/\varepsilon)}$ (solving for $T$ in $\varepsilon$). Now recalling the definition of $g$, and the fact that the supremum over $g \in \mathcal{G}$ is greater than the supremum over $f \in \mathcal{F}$ for a function of $f - f_n^O$, with $f_n^O$ fixed, we have
	\begin{align*}
		\langle \epsilon, g\rangle_n &= \langle \epsilon, f - f_n^O\rangle_n \{P(f) - P(f_n^O)\}^{-1} \\
		\norm{g}_n^{1-\alpha/2} &= \norm{f - f_n^O}_n^{1-\alpha/2} \{P(f) - P(f_n^O)\}^{\alpha/2-1}.
	\end{align*}
	Therefore, we have
	\begin{align*}
		P \left[\sup_{f \in \mathcal{F}} \frac{\left|\langle \epsilon, f - f_n^O\rangle_n\right|}{\norm{f - f_n^O}_{n}^{1-\alpha/2}\{P(f) - P(f_n^O)\}^{\alpha/2}}  \leq n^{-1/2} C_{\varepsilon} \right] \geq 1 - \varepsilon.
	\end{align*}
\end{proof}

Lemma \ref{lemma:vandegeer} is the last piece which allows us to prove Theorem 1. The outline of Theorem 1 is that we use the expression from Lemma \ref{lemma:2.1}, apply Lemma \ref{lemma:vandegeer} to deal with the $\epsilon$ terms, then apply several further inequalities. Towards the end of the proof, we see that our choice of $\lambda_n$ allows us to have the simplified expression that we present in the statement of Theorem 1.

\begin{manualtheorem}{1}[Theorem 2.2 of \citet{simonConvergenceRatesNonparametric2021}]
	Suppose data are generated according to (1), $\hat{f}$ is defined as in (3) for some $\lambda_n > 0$ and $\mathcal{F}_k$ defined as in (4). Let $f_0, f_1, ... \in \mathcal{F}_k$ with $P_k(f_j) > 0$ for all $j$.
	If we choose
 \begin{align*}
     \lambda_n = O_p\{n^{-2/(2+k^{-1})}P_k^{(k^{-1}-2)/(k^{-1}+2)}(f_n)\},
 \end{align*}
 then
\begin{align}\label{inequality}
||\hat{f} - f^*||_n^2 \leq ||f^* - f_n||_n^2 + O_p\{\lambda_n P_k(f_n)\}.
\end{align}
\end{manualtheorem}
\begin{proof}
	From Lemma \ref{lemma:2.1}, we have the following upper bound:
	\begin{align*}
		\norm{\hat{f} - f^*}_n^2 + \norm{\hat{f} - f^O}_n^2 + 2\lambda P(\hat{f}) \leq \norm{f^O - f^*}_n^2 + 2\langle \epsilon, \hat{f} - f^O\rangle_n + 2\lambda P(f^O).
	\end{align*}
	Now we can apply Lemma \ref{lemma:vandegeer} to deal with the $2\langle \epsilon, \hat{f} - f^O\rangle_n$ term, with $\alpha = 1/k$ \citep{mammenNonparametricRegressionQualitative1991, geerEmpiricalProcessesMEstimation2009}. Specifically, using Lemma \ref{lemma:vandegeer}, with $f_n^O \equiv f^O$ and choosing $\hat{f}$ as the $f \in \mathcal{F}$ (since the probability in Lemma \ref{lemma:vandegeer} will be at least as large if we focus on one $f$ rather than the supremum over $f \in \mathcal{F}$), we have with at least probability $1 - \epsilon$,
	\begin{align*}
		\langle \epsilon, \hat{f} - f^O_n\rangle_n \leq C_\epsilon n^{-1/2} \norm{\hat{f} - f_n^O}_n^{1-\alpha/2} \left\{P(\hat{f}) + P(f_n^O) \right\}^{\alpha/2}.
	\end{align*}
	Therefore, with at least probability $1 - \epsilon$:
	\begin{align}
		&\norm{\hat{f} - f^*}_n^2 + \norm{\hat{f} - f^O}_n^2 + 2\lambda P(\hat{f}) \leq \nonumber\\
		&\hspace{1in} \norm{f^O - f^*}_n^2 + 2 \underbrace{C_\epsilon n^{-1/2} \norm{\hat{f} - f_n^O}_n^{1-\alpha/2} \left\{P(\hat{f}) + P(f_n^O) \right\}^{\alpha/2}}_{(I)} + 2\lambda P(f^O).\label{eq:fullineq}
	\end{align}
	Now we focus on upper bounding the new term $(I)$ from Lemma \ref{lemma:vandegeer} using Young's Inequality. Recall that Young's Inequality says that $x y \leq x^a / a + y^b / b$ when $1 / a + 1/b = 1$. If we let
	\begin{align*}
		x &= \norm{\hat{f} - f_n^O}_n^{1-\alpha/2} \\
		y &= C_\epsilon n^{-1/2} \left\{P(\hat{f}) + P(f_n^O) \right\}^{\alpha/2} \\
		a &= 4/(2-\alpha) \\
		b &= 4/(2+\alpha)
	\end{align*}
	then applying Young's Inequality, we have
	\begin{align*}
		(I) \equiv x y &\leq x^a/a + y^b/b \\
		&= \frac{\left(\norm{\hat{f} - f_n^O}_n^{1-\alpha/2}\right)^{4/(2-\alpha)}}{4/(2-\alpha)} + \frac{\left[C_\epsilon n^{-1/2} \left\{P(\hat{f}) + P(f_n^O) \right\}^{\alpha/2}\right]^{4/(2+\alpha)}}{4/(2+\alpha)} \\
		&\leq \frac{1}{2}\norm{\hat{f} - f_n^O}_n^{2} + C_\epsilon'n^{-2/(2+\alpha)} \left\{P(\hat{f}) + P(f_n^O) \right\}^{2\alpha/(2+\alpha)}.
	\end{align*}
	This was a convenient choice of $x$ and $y$ because we now have cancellation for the $\norm{\hat{f} - f_n^O}_n^{2}$ term on the LHS and RHS of \eqref{eq:fullineq} to obtain
	\begin{align}\label{eq:simp-ineq}
		\norm{\hat{f} - f^*}_n^2 + 2\lambda P(\hat{f}) \leq \norm{f^O - f^*}_n^2 + C''_{\epsilon} n^{-2/(2+\alpha)} \left\{P(\hat{f}) + P(f_n^O) \right\}^{2\alpha/(2+\alpha)} + 2\lambda P(f_n^O).
	\end{align}
	Our final step is to take care of $P(\hat{f})$ and $P(f_n^O)$, once we choose $\lambda \equiv \lambda_n$ appropriately as a function of $n$.
	First consider the case where $P(\hat{f}) \leq P(f_n^O)$. We have, choosing $\lambda_n = O_P(n^{-2/(2+\alpha)} P^{-(2-\alpha)/(2+\alpha)}(f_n^O))$
	\begin{align*}
		\norm{\hat{f} - f^*}_n^2 &\leq \norm{f^O - f^*}_n^2 + C''_{\epsilon} n^{-2/(2+\alpha)} \left[P(\hat{f}) + P(f_n^O) \right]^{2\alpha/(2+\alpha)} + 2\lambda_n P(f_n^O) \\
		&\leq \norm{f^O - f^*}_n^2 + C'''_{\epsilon} n^{-2/(2+\alpha)} P^{2\alpha/(2+\alpha)}(f_n^O) + 2\lambda_n P(f_n^O) \\
		&= \norm{f^O - f^*}_n^2 + C'''_{\epsilon} \big\{\underbrace{n^{-2/(2+\alpha)} P^{-(2-\alpha)/(2+\alpha)}(f_n^O)}_{O_P(\lambda_n)}\big\}  P(f_n^O) + 2\lambda_n P(f_n^O) \\
		&= \norm{f^O - f^*}_n^2 + O_P\{\lambda_n P(f^O_n)\} + 2\lambda_n P(f^O_n) \\
		&\leq \norm{f^O - f^*}_n^2 + O_P\{\lambda_n P(f^O_n)\}.
	\end{align*}
	Now consider the case where $P(\hat{f}) \geq P(f_n^O)$.
 In this case, we need to upper bound $C''_{\epsilon} n^{-2/(2+\alpha)} \left\{P(\hat{f}) + P(f_n^O) \right\}^{2\alpha/(2+\alpha)}$ using only $P(f_n^O)$ and some function of $\lambda_n$. We can again use Young's Inequality with
	\begin{align*}
		x &= \left\{P(\hat{f}) + P(f_n^O) \right\}^{2\alpha/(2+\alpha)} (2\lambda_n)^{2\alpha/(2+\alpha)}\\
		y &= C''_{\epsilon} n^{-2/(2+\alpha)} (2\lambda_n)^{-2\alpha/(2+\alpha)}\\
		a &= (2+\alpha)/2\alpha\\
		b &= (2+\alpha)/(2-\alpha)
	\end{align*}
	to obtain the following:
	\begin{align*}
		xy \equiv C''_{\epsilon} n^{-2/(2+\alpha)} \left\{P(\hat{f}) + P(f_n^O) \right\}^{2\alpha/(2+\alpha)} &\leq \frac{\left[\left\{P(\hat{f}) + P(f_n^O) \right\}^{2\alpha/(2+\alpha)} (2\lambda_n)^{2\alpha/(2+\alpha)}\right]^{(2+\alpha)/2\alpha}2\alpha}{(2+\alpha)} \\
		&\quad\quad\quad\quad + \frac{\left\{C''_{\epsilon} n^{-2/(2+\alpha)} (2\lambda_n)^{-2\alpha/(2+\alpha)}\right\}^{(2+\alpha)/(2-\alpha)}(2-\alpha)}{(2+\alpha)} \\
		&= 2\lambda_n \left\{P(\hat{f}) + P(f_n^O) \right\}(2\alpha)/(2+\alpha) \\
		&\quad\quad\quad\quad + C'''_{\epsilon}n^{-2/(2-\alpha)}\lambda_n^{-2\alpha/(2-\alpha)} (2-\alpha)/(2+\alpha) \\
		&\leq 2\lambda_n \left\{P(\hat{f}) + P(f_n^O) \right\} + C'''_{\epsilon}n^{-2/(2-\alpha)}\lambda_n^{-2\alpha/(2-\alpha)}.
	\end{align*}
	The $2\lambda_n P(\hat{f})$ now cancels out on both sides of \eqref{eq:simp-ineq}. Furthermore,
	\begin{align*}
		n^{-2/(2-\alpha)}\lambda_n^{-2\alpha/(2-\alpha)} &= n^{-2/(2-\alpha)}\left\{n^{-2/(2+\alpha)} P^{-(2-\alpha)/(2+\alpha)}(f_n^O) \right\}^{2\alpha/(\alpha-2)} \\
		&= n^{-2/(2-\alpha) -4\alpha/\{(2+\alpha)(\alpha-2)\}} P^{2\alpha/(2+\alpha)}(f_n^O) \\
		&= n^{-2/(2+\alpha)} P^{-(2-\alpha)/(2+\alpha)}(f_n^O)  P(f_n^O).
	\end{align*}
	Therefore, $C'''_{\epsilon}n^{-2/(2-\alpha)}\lambda_n^{-2\alpha/(2-\alpha)} = O_P\{\lambda_n P(f_n^O)\}$, and \eqref{eq:simp-ineq} becomes
	\begin{align*}
		\norm{\hat{f} - f^*}_n^2  \leq \norm{f^O - f^*}_n^2 + O_P\{\lambda_n P(f_n^O)\}.
	\end{align*}
	This completes the proof.

\end{proof}

\newpage
\section{Representation of Functions as Piecewise Polynomials}\label{app:birman}

The purpose of this section is to work with piecewise constant approximations to functions with bounded total variation. Lemma \ref{f-epsilon} is directly from \citet{birmanPIECEWISEPOLYNOMIALAPPROXIMATIONSFUNCTIONS1967}, which we include here for clarity. Presumably, the result of Lemma \ref{lemma:penalty-approx} was known to the authors of \citet{birmanPIECEWISEPOLYNOMIALAPPROXIMATIONSFUNCTIONS1967}, but we did not find a clear proof, so we include our own here. We now provide some justification of how the following lemmas will be helpful.

Although we define our approximation function $f^O$ as the convolution of $f^*$ with a higher-order kernel, upper bounding the approximation error between $f^O$ and $f^*$ is much easier if we first decompose it into the approximation error between $f^*$ and a piecewise polynomial approximation to $f^*$ (which we later term $f_{\epsilon}$), and the approximation error between $f_{\epsilon}$ and $f_{\epsilon}$ convolved with the same higher order kernel. We show how to perform this decomposition in Lemma \ref{approx-error}. When working with the approximation error between $f_{\epsilon}$ and $f_{\epsilon}$ convolved with a higher order kernel, we will find in Lemma \ref{approx-error} that it is a function of the $\ell$th-order (with $f^* \in \mathcal{F}_{\ell}$) total variation of $f_{\epsilon}$. Using Lemma \ref{lemma:penalty-approx}, we can claim that this is no larger than that of the true function $f^*$. 

Lemma \ref{f-epsilon} tells us about the existence of piecewise constant approximations (we use this to approximate the \textit{$(\ell-1)$th order derivative} of $f^*$), which we can then integrate to obtain an $(\ell-1)$th-order piecewise polynomial approximation to $f^*$.

\begin{lemma}[Theorem 3.1 of \citet{birmanPIECEWISEPOLYNOMIALAPPROXIMATIONSFUNCTIONS1967}]\label{f-epsilon}
	Let $f \in \mathcal{F}_{1}$. Then for any $\epsilon > 0$ and $\omega = 1$, there exists a partition $\Xi$ of $\mathcal{X} \equiv [0, 1]$ such that $|\Xi| \leq \epsilon^{-1/\omega}$, and
	\begin{align*}
		||f - f_{\epsilon}||_{\infty} \leq C \epsilon P_{1}(f)
	\end{align*}
	where $f_{\epsilon}$ is a piecewise constant function with partitions $\Xi$.
\end{lemma}
\begin{proof} See Theorem 3.1 of \citet{birmanPIECEWISEPOLYNOMIALAPPROXIMATIONSFUNCTIONS1967}, where in our case, $p = 1$, $\alpha = 1$, $m = 1$, $\epsilon = n^{-\omega}$, and $\omega \equiv \alpha m^{-1} = 1$. Note that when $p = 1$, their $||f||_{L_p^{\alpha}} = P_{\alpha}(f)$.
\end{proof}

\begin{lemma}\label{lemma:penalty-approx}
	Let $f_{\epsilon}$ be the piecewise constant approximation to $f \in \mathcal{F}_{1}$ given in Lemma \ref{f-epsilon}. Then $P_{1}(f_{\epsilon}) \leq P_{1}(f)$.
\end{lemma}
\begin{proof}
	We aim to show that the total variation of $f_{\epsilon}$ is no more than that of $f$, the function that $f_{\epsilon}$ is approximating, i.e., that
	\begin{align*}
		\sup_{x_1, ..., x_M} \sum_{m=1}^{M} |f_{\epsilon}(x_{m+1}) - f_{\epsilon}(x_{m})| \leq \sup_{x_1, ..., x_M} \sum_{m=1}^{M} |f(x_{m+1}) - f(x_m)|.
	\end{align*}
	We know that $f_{\epsilon}(x)$ can be written as $f_{\epsilon} = \beta_0 + \sum_{j=1}^{J} \beta_j 1(x > x_j)$. Furthermore, based on the construction of the partitions and approximating function $f_{\epsilon}$ in \citet{birmanPIECEWISEPOLYNOMIALAPPROXIMATIONSFUNCTIONS1967},
	\begin{align}
			\int_{x_j}^{x_{j+1}} f_{\epsilon}(x)dx = \int_{x_j}^{x_{j+1}} f(x)dx.\label{partition-integral}
		\end{align}
		First consider the case where $f$ is a continuous function. Define $\mathcal{X}_1$ as the partition that defines each of the piecewise constant jumps in $f_{\epsilon}$, i.e., $\mathcal{X}_1 = \{x_1, ..., x_J\}$, where $x_{J+1} = 1$. This partition achieves the supremum over all finite partitions for $f_{\epsilon}$ because it is exactly the total variation (the function only changes at each of these piecewise constant jumps). We know that $f$ is continuous, and we also know that since $f_{\epsilon}$ is constant within $(x_{j}, x_{j+1}]$, there must exist points $(x^-_{j}, x^+_{j}) \in (x_{j}, x_{j+1}]$ such that $f(x^-_{j}) \leq f_{\epsilon}(x^-_{j})$ and $f(x^+_{j}) \geq f_{\epsilon}(x^+_{j})$. By the intermediate value theorem, there exists a point $x_{j}^* \in [x^-_{j}, x^+_{j}]$ such that $f(x_j^*) = f_{\epsilon}(x_j^*)$. Therefore, if we construct the partition $\mathcal{X}_1^* := \{x_1^*, ..., x_J^*\}$, the variation of $f$ based on this partition is exactly equal to that of $f_{\epsilon}$ on $\mathcal{X}_1$. The total variation is defined as the supremum over all finite partitions, and we have found a partition for $f$ such that the variation is equal to the supremum over all partitions for $f_{\epsilon}$; thus, $P_1(f) \geq P_1(f_{\epsilon})$.
	
	Now consider the case where $f$ is not necessarily continuous, where the intermediate value theorem no longer applies. However, because of how the partitions have been constructed, we still know that in order for \eqref{partition-integral} to hold, within each partition, $(x_{j}, x_{j+1}]$, there must exist points $x^-_{j}, x^+_{j} \in (x_{j}, x_{j+1}]$ such that $f(x^-_{j}) \leq f_{\epsilon}(x^-_{j})$ and $f(x^+_{j}) \geq f_{\epsilon}(x^+_{j})$. Heuristically, we will use these to construct a new partition for $f$ where the function values of $f$ ``overshoot'' or ``undershoot'' the values of $f_{\epsilon}$ whenever $f_{\epsilon}$ changes direction.

	Define the sets
	\begin{align*}
		\mathcal{X}^+ &:= \{ x^+_{j} \in \mathcal{X}_1: f_{\epsilon}(x_j) \geq \max\{f_{\epsilon}(x_{j-1}), f_{\epsilon}(x_{j+1})\}\} \\
		\mathcal{X}^- &:= \{ x^{-}_{j} \in \mathcal{X}_1: f_{\epsilon}(x_j) \leq \min\{f_{\epsilon}(x_{j-1}), f_{\epsilon}(x_{j+1})\}\}
	\end{align*}
	which contain all of the partitions in $\mathcal{X}_1$ where $f_{\epsilon}$ changes direction. We define the partition set $\mathcal{X}_2 = \{\mathcal{X}^+ \cup \mathcal{X}^-\}$. We now need to deal with the endpoints. If $f_{\epsilon}(x_2) - f_{\epsilon}(x_1) > 0$, then we add $x^-_1$ to $\mathcal{X}_2$, otherwise, we add $x^+_1$. Conversely, if $f_{\epsilon}(x_{J+1}) - f_{\epsilon}(x_{J}) > 0$, we add $x^+_{J+1}$ to $\mathcal{X}_2$, otherwise we add $x^-_{J+1}$.
	
	We will refer to the items of  $\mathcal{X}_2$ as $x'$. Based on how we have defined the partitions in $\mathcal{X}_2$, for all $i \in 1, ..., |\mathcal{X}_2| - 1$, we know that $|f_{\epsilon}(x'_{i+1}) - f_{\epsilon}(x'_{i})| \leq |f(x'_{i+1}) - f(x'_i)|$. Furthermore, $f_{\epsilon}$ has the same variation in $\mathcal{X}_2$ as it does in $\mathcal{X}_1$ because within neighboring points $x'_i, x'_{i+1} \in \mathcal{X}_2$, $f_{\epsilon}$ is a monotonic step function. This means that the $\mathcal{X}_2$ partition achieves the supremum over all finite partitions for $f_{\epsilon}$. Therefore,
	\begin{align*}
		P_1(f_{\epsilon}) &:= \sum_{i=1}^{|\mathcal{X}_2| - 1} |f_{\epsilon}(x'_{i+1}) - f_{\epsilon}(x'_{i})| \\
		&\leq \sum_{i=1}^{|\mathcal{X}_2| - 1} |f(x'_{i+1}) - f'(x_{i})| \leq P_1(f),
	\end{align*}
	where the last inequality comes from the fact that $\mathcal{X}_2$ is only one possible partition, so it must be less than or equal to the supremum over all partitions. Thus, $P_1(f_{\epsilon}) \leq P_1(f)$.
\end{proof}

\newpage
\section{Improved Convergence Rates}\label{app:new}

\subsection{Higher-Order Kernels}

We define our approximation function as the convolution of the true function $f^*$ with a higher-order kernel. We will work with the scaled kernel, $H_{k,\delta}: u \to (1/\delta)H_k(u/\delta)$ such that it is supported on $[-\delta, \delta]$. Let $f_{\delta,k}^O$ be the convolution of the true function $f^*$ and a $k$th-order kernel as defined as in Proposition \ref{kernels} below:
\begin{align*}
	f_{\delta,k}^O := f^* \star H_{k,\delta} \equiv x &\to \int_{-\infty}^{\infty} f^*(x-t) H_{k, \delta}(t) dt \equiv \int_{-\delta}^{\delta} f^*(x-t) H_{k, \delta}(t) dt.
 \end{align*}
 Using this approximation function, we will derive upper bounds on both the approximation error and the penalty function. Our results are summarized in Lemma \ref{approx-error} and \ref{penalty}, which we then use in Theorem \ref{main-thm}. Proposition \ref{kernels} is justified using Section 1.2.2, and specifically Proposition 1.3 in \citet{tsybakov2009}. Note that depending on the order of kernel needed in the proof, the kernel given by Proposition \ref{kernels} may necessarily have both negative and positive values.

\begin{proposition}[Existence of Higher-Order Kernels]\label{kernels}
	For any integer $k$, there exists a symmetric function $H_k$ such that $H_k$
	\begin{itemize}
		\item Has bounded support on $[-1,1]$
		\item Integrates to 1: $\int H_k(u) du = 1$
		\item Is bounded and has bounded derivatives: $\sup_{u\in[-1,1]}|H_k^{(\ell)}(u)| < C_{k,\ell}$ for some constant $C_{k,\ell}$ and $\ell = 0, 1, ..., k-1$.
		\item $\int u^{\ell} H_k(u) du = 0$ for all $\ell < k$
		\item $\int |u|^k H_k(u) du \leq C_k$ for some constant $C_k$.
	\end{itemize}
\end{proposition}

\subsection{Working with Piecewise Polynomials and Convolutions}

In this section, we provide some lemmas which allow us to prove our main results on the approximation error in Lemma 1 and penalty function in Lemma 2.  The key results that we will use from this section is that the convolution of a higher-order kernel with a polynomial of lesser order returns exactly the polynomial we started with. This is used when we decompose our approximation error into the approximation error between $f_{\epsilon}$ (a piecewise polynomial approximation to $f^*$), and the convolution of $f_{\epsilon}$ with a higher-order kernel. We also derive the $k$th order total variation of a convolution of a function $f \in \mathcal{F}_{\ell}$ with a $k$th-order kernel ($k > \ell$), and show that it is a function of the $\ell$th-order total variation of the original function, the bandwidth of the higher-order kernel $\delta$ and the order of the kernel $k$ relative to $\ell$.

\begin{lemma}[Representation of a Piecewise Polynomial]\label{f-epsilon-rep}
	If $f$ is an $a$th order piecewise polynomial with continuous derivatives up to order $(a - 1)$th order, it can be represented as
	\begin{align*}
		f(x) = f_0(x) + \sum_{j=1}^{J} \beta_{j} (x - d_j)^{a} 1(x - d_j \geq 0)
	\end{align*}
	where $f_0(x)$ is an $a$th order polynomial with continuous derivatives up to order $(a - 1)$ as well, and where $J$ is the number of partitions.
\end{lemma}
\begin{proof}
    Consider a general $a$th-order piecewise polynomial function,
    \begin{align*}
        f(x) = f_0(x) + \sum_{j=1}^J \sum_{m=0}^{a} \beta_{j,m} (x - d_j)^{m} 1(x - d_j \geq 0).
    \end{align*}
    where $f_0(x)$ is an $a$th-order polynomial with continuous derivatives up to order $a - 1$.
    Let $b < a$. Then,
    \begin{align*}
        f^{(b)}(x) = f^{(b)}_0(x) + \sum_{j=1}^{J} \sum_{m=b}^{a} \beta_{j,m}(x - d_j)^{m-b} 1(x - d_j \geq 0) (m)!/(m-b)!.
    \end{align*}
    Consider the case where $b = a - 1$. In order for $f^{(a-1)}(x)$ to be continuous, all $\beta_{j,a-1}$ must be zero, otherwise there could be discontinuities. We can continue this exercise with $b = a - 2$, all the way to $b = 0$, noting each time that the next group of coefficients must also be zero for the same reasoning.
    Therefore, the only non-zero coefficients allowed are $\beta_{j,a}$.
\end{proof}

\begin{lemma}[Convolution with a Polynomial]\label{poly-conv}
	Let $h(x) = \beta_0 x^{\ell}$, and $k > \ell$. Then $h = h \circ H_k$, where $H_k$ is defined in Proposition \ref{kernels}.
\end{lemma}
\begin{proof}
	We can prove this using the binomial expansion:
	\begin{align*}
		h(x) - h(x_0) &= \beta_0(x^\ell - x_0^\ell) - \beta_0(x - x_0)^\ell + \beta_0(x - x_0)^\ell \\
		&= \beta_0(x - x_0)^\ell - \beta_0 \sum_{j=1}^{\ell-1} {\ell\choose j} x^{\ell - j} (-x_0)^j.
	\end{align*}
	Using the definition of convolution, and the expansion above setting $x_0 = t$:
	\begin{align*}
		(h \circ H_k)(x) &= \int_{-\infty}^{\infty} h(x - t) H_k(t) dt \\
		&= \int_{-\infty}^{\infty} \beta_0 (x - t)^{\ell} H_k(t) dt \\
		&= \beta_0 x^{\ell} \int_{-\infty}^{\infty} H_k(t) dt - \beta_0\int_{-\infty}^{\infty}  t^{\ell} H_k(t) dt + \beta_0 \sum_{j=1}^{\ell-1} {\ell\choose j}x^{\ell - j} (-1)^j \int t^{j} H_k(t) dt.
	\end{align*}
	Since $H_k$ is a kernel that satisfies Proposition \ref{kernels}, all of the integrals that have $t^{j} H_k(t)$, $0 < j \leq \ell$ are zero, and $\int_{-\infty}^{\infty} H_k(t) dt = 1$. Therefore, $(h \circ H_k)(x) = h(x)$.
\end{proof}

\begin{lemma}[$k$th Order Total Variation of a Convolution]\label{lemma:tv-convolution}
Let $f \in \mathcal{F}_{\ell}$, i.e., that it has bounded $\ell$th order total variation. Let $f^O_{\delta,k} = f \star H_{\delta,k}$, with $k > \ell$. Then
	\begin{align*}
	P_k(f_{\delta,k}^O) \leq 2\left\{(1/\delta)^{(k-\ell)}C_{k,k-\ell} \right\}P_{\ell}(f).
\end{align*}

\end{lemma}
\begin{proof}
	An alternative definition of the $k$th order total variation of $f$ is
\begin{align}\label{total-variation}
	P_k(f) \equiv \sup_{x_1, ..., x_M} \sum_{m=1}^{M} \left|f^{(k-1)}(x_{m+1}) - f^{(k-1)}(x_{m})\right|
\end{align}
where $(x_1, ..., x_M)$ is a finite partition. Based on the definition of $H_{k,\delta}$, we have that
\begin{align}
	H_{k,\delta}^{(\ell-1)}(t) &= \left\{(1/\delta)H_k(t/\delta)\right\}^{(\ell-1)}\nonumber \\
	&= (1/\delta) (1/\delta)^{\ell-1} H^{(\ell-1)}_k(t/\delta)\nonumber \\
	&= (1/\delta)^{(\ell)} H^{(\ell-1)}_k(t/\delta).\label{eq:h-deriv}
\end{align}
Then, using the definition of $f_{\delta,k}^O$ and the fact that for any $k_1 + k_2 = k$, $(f \star g)^{(k)} = f^{(k_1)} \star f^{(k_2)}$, we can write, for any $(x, y)$ pair,
\begin{align}
	(f_{\delta,k}^O)^{(k-1)}(x) - (f_{\delta,k}^O)^{(k-1)}(y) &= (f \star H_{k,\delta})^{(k-1)}(x) - (f \star H_{k,\delta})^{(k-1)}(y) \nonumber\\
	&= \left\{f^{(\ell-1)} \star H_{k,\delta}^{(k-\ell)}\right\}(x) - \left\{f^{(\ell-1)} \star H_{k,\delta}^{(k-\ell)}\right\}(y) \nonumber\\
	&= \int_{-\delta}^{\delta} \left\{f^{(\ell-1)}(x-t) - f^{(\ell-1)}(y-t)\right\}H_{k,\delta}^{(k-\ell)}(t) dt \nonumber\\
	\text{(applying \eqref{eq:h-deriv})} &= \left(1/\delta^{k-\ell-1}\right)\int_{-\delta}^{\delta} \left[f^{(\ell-1)}(x-t) - f^{(\ell-1)}(y-t)\right]H^{(k-\ell)}_{k}(t/\delta) dt \nonumber\\
	\text{(letting $s = t/\delta$)} &= \left(1/\delta^{k-\ell}\right)\int_{-1}^{1} \left[f^{(\ell-1)}(x-\delta s) - f^{(\ell-1)}(y-\delta s)\right]H^{(k-\ell)}_{k}(s) ds\label{eq:conv-deriv}
\end{align}
Taking the absolute value of \eqref{eq:conv-deriv}, we get
\begin{align*}
	\left|(f_{\delta,k}^O)^{(k-1)}(x) - (f_{\delta,k}^O)^{(k-1)}(y)\right| 
	&= \left(1/\delta^{k-\ell}\right) \left|\int_{-1}^{1} \left\{f^{(\ell-1)}(x-\delta s) - f^{(\ell-1)}(y-\delta s)\right\}H^{(k-\ell)}_{k}(s) ds\right| \\
	&\leq \left(1/\delta^{k-\ell}\right) \int_{-1}^{1} \left|f^{(\ell-1)}(x-\delta s) - f^{(\ell-1)}(y-\delta s)\right|\left|H^{(k-\ell)}_{k}(s)\right| ds \\
	&\leq \left(1/\delta^{k-\ell}\right) C_{k,k-\ell}\int_{-1}^{1} \left|f^{(\ell-1)}(x-\delta s) - f^{(\ell-1)}(y-\delta s)\right| ds
\end{align*}
where $C_{k,k-\ell}$ is the constant bounding the derivative of $H_k$ from Proposition \ref{kernels}. Now using the definition in \eqref{total-variation}, and setting $x = x_{m+1}$ and $y = x_m$, we have
\begin{align*}
	P_k(f_{\delta,k}^O) &\equiv \sup_{x_1, ..., x_M} \sum_{m=1}^{M} \left|(f_{\delta,k}^O)^{(k-1)}(x_{m+1}) - (f_{\delta,k}^O)^{(k-1)}(x_{m})\right| \\
	&\leq \sup_{x_1, ..., x_M} \sum_{m=1}^{M} \left(1/\delta^{k-\ell}\right) C_{k,k-\ell}\int_{-1}^{1} \left|f^{(\ell-1)}(x_{m+1}-\delta s) - f^{(\ell-1)}(x_{m}-\delta s)\right| ds. \\
	&\leq \left[(1/\delta)^{(k-\ell)}C_{k,k-\ell} \right]\sup_{x_1, ..., x_M} \sum_{m=1}^{M} \int_{-1}^{1} \left|f^{(\ell-1)}(x_{m+1}-\delta s) - f^{(\ell-1)}(x_{m}-\delta s)\right| ds.
\end{align*}
We can now switch the integral and the finite summation $(M < \infty)$, and further upper bound the supremum of the integral by the integral of the supremum:
\begin{align*}
	P_k(f_{\delta,k}^O) \leq \left\{(1/\delta)^{(k-\ell)}C_{k,k-\ell} \right\}\int_{-1}^{1} \sup_{x_1, ..., x_M} \sum_{m=1}^{M} \left|(f^*)^{(\ell-1)}(x_{m+1}-\delta s) - (f^*)^{(\ell-1)}(x_{m}-\delta s)\right| ds.
\end{align*}
Finally, since the supremum considers all finite partitions, we can remove the $\delta s$ shift in the $x's$, to get an expression with the $\ell$th order total variation of $f^*$:
\begin{align*}
	P_k(f_{\delta,k}^O) &\leq \left\{(1/\delta)^{(k-\ell)}C_{k,k-\ell} \right\}\int_{-1}^{1} \sup_{x_1, ..., x_M} \sum_{m=1}^{M} \left|f^{(\ell-1)}(x_{m+1}) - f^{(\ell-1)}(x_{m})\right| dt \\
	&= \left\{(1/\delta)^{(k-\ell)}C_{k,k-\ell} \right\}\int_{-1}^{1} P_{\ell}(f) dt \\
	&= 2\left\{(1/\delta)^{(k-\ell)}C_{k,k-\ell} \right\}P_{\ell}(f).
\end{align*}
\end{proof}

\subsection{Proofs of Main Text Results}

In this section, we derive the approximation error and penalty function using our approximation function, which is the convolution of $f^*$ with a higher-order kernel. These are the key components that we can plug into the MSE upper bounding equation of Theorem 1 to obtain our final result on rates in Theorem \ref{main-thm}.

\begin{manuallemma}{1}[Upper Bound on Approximation Error]\label{approx-error}
	Let $f^* \in \mathcal{F}_{\ell}$, and $k > \ell$. Then
\begin{align*}
    ||f^* - f_{\delta,k}^O||_n^2 = O_p(\delta^{2\ell-1}).
\end{align*}
\end{manuallemma}
\begin{proof}
Our goal is to upper bound $||f^* - f_{\delta,k}^O||_n$. First, we define $f_{\epsilon}^{(\ell-1)}$ as the piecewise constant approximation to the $(\ell-1)$th derivative of $f^*$, since we know that $f^{*(\ell-1)} \in \mathcal{F}_{1}$, that exists by Lemma \ref{f-epsilon}. Furthermore, we can define $f^{(\ell-1)}_{\epsilon}(x) = \int_{-\infty}^{0} f^{*(\ell-1)}(t) dt$ for all $x \in [-1, 0)$ and $f^{(\ell-1)}_{\epsilon}(x) = 0$ for $x < -1$. This way the two functions have the same integral in the left tail. Since $f_{\epsilon}^{(\ell-1)}$ is a piecewise constant, $f_{\epsilon}$ is an $(\ell-1)$th order piecewise polynomial with continuous derivatives up to order $\ell - 2$. We use the following expansion: 
\begin{align}
	||f^* - f_{\delta,k}^O||_n &\leq ||f^* - f_{\epsilon}||_n + ||f_{\epsilon} - f_{\epsilon} \star H_{\delta,k}||_n + ||f_{\epsilon} \star H_{\delta,k} - f^* \star H_{\delta,k}||_n\nonumber \\
	&\leq ||f^* - f_{\epsilon}||_{\infty} + ||f_{\epsilon} - f_{\epsilon} \star H_{\delta,k}||_n + ||(f_{\epsilon} - f^*) \star H_{\delta,k}||_{\infty}\nonumber \\
	\text{(Young's Convolution Inequality)} &\leq ||f^* - f_{\epsilon}||_{\infty} + ||f_{\epsilon} - f_{\epsilon} \star H_{\delta,k}||_n + ||f_{\epsilon} - f^*||_{\infty} ||H_{\delta,k}||_{L^1}\nonumber \\
	\text{(properties of $H_{\delta,k}$ from Proposition \ref{kernels})}&\leq ||f^* - f_{\epsilon}||_{\infty} + ||f_{\epsilon} - f_{\epsilon} \star H_{\delta,k}||_n + ||f_{\epsilon} - f^*||_{\infty} C_{k,0} \nonumber \\
	&\leq 2C'||f^* - f_{\epsilon}||_{\infty} + ||f_{\epsilon} - f_{\epsilon} \star H_{\delta,k}||_n\nonumber.
\end{align}
where $L^1(H_{\delta,k}) = \int |H_{\delta,k}(x)|dx = \int |H_k(x)| dx \leq C_{k,0}$ since $H_{k}$ is bounded and has bounded support on $[-1, 1]$, and $C' = \max(1, C_{k,0})$. Therefore, we would like to upper bound both $||f^* - f_{\epsilon}||_{\infty}$ and $||f_{\epsilon} - f_{\epsilon} \star H_{\delta,k}||_n$. 

First we focus on $||f^* - f_{\epsilon}||_{\infty}$. By Lemma \ref{f-epsilon}, $||f^{*(\ell-1)} - f_{\epsilon}^{(\ell-1)}||_{\infty} \leq C'' \epsilon P_{1}(f^{*(\ell-1)})$ for some $C''$. Therefore,
	\begin{align*}
		||f^* - f_{\epsilon}||_{\infty} = \left|\left|\int_{-\infty}^{x} \dots \int_{-\infty}^{t_2} [f^{*(\ell-1)}(t_1) - f^{(\ell-1)}_{\epsilon}(t_1)] dt_1 \dots dt_{\ell-1}\right|\right|_{\infty}.
	\end{align*}
	Since we have defined the functions $f_{\epsilon}^{(\ell-1)}$ and $f^{*(\ell-1)}$ to have equivalent integrals in the left tails, this simplifies to
	\begin{align*}
		||f^* - f_{\epsilon}||_{\infty} &= \left|\left|\int_{0}^{x} \dots \int_{0}^{t_2} [f^{*(\ell-1)}(t_1) - f^{(\ell-1)}_{\epsilon}(t_1)] dt_1 \dots dt_{\ell-1}\right|\right|_{\infty} \\
		&\leq \int_0^{x} \dots \int_0^{t_2} \left|\left|f^{*(\ell-1)} - f^{(\ell-1)}_{\epsilon}\right|\right|_{\infty} dt_1 \dots dt_{\ell-1} \\
		&\leq \left|\left|f^{*(\ell-1)} - f^{(\ell-1)}_{\epsilon}\right|\right| \int_0^1 \dots \int_0^1 dt_1 \dots dt_{\ell-1} \\
		&\leq C'' \epsilon P_1(f^{*(\ell-1)}) \\
		&= C''\epsilon P_{\ell}(f^*).
	\end{align*}

Now we turn to $||f_{\epsilon} - f_{\epsilon} \star H_{\delta,k}||_n$. Using the representation of a piecewise polynomial in Lemma \ref{f-epsilon-rep} (and with $f_{\epsilon}$ an $(\ell - 1)$th order piecewise polynomial), Lemma \ref{poly-conv} (using $\ell - 1$ rather than $\ell$), and the fact that convolution is a linear operator,
\begin{align*}
	(f_{\epsilon} \star H_{\delta,k}) = f_0(x) + \sum_{j=1}^{J(\epsilon)} \beta_{j} \big[\big\{(\cdot - d_j)^{\ell-1} 1(\cdot  \geq d_j) \big\}\star H_{\delta, k}\big](x).
\end{align*}
Therefore, we can express the empirical 2-norm difference between $f_{\epsilon} \star H_{\delta,k}$ and $f_{\epsilon}$ as
\begin{align}
	||f_{\epsilon} - (f_{\epsilon} \star H_{\delta,k})||_n &= \left\|\sum_{j=1}^{J(\epsilon)} \beta_{j} \Big[\big\{(x - d_j)^{\ell-1} 1(x  \geq d_j) \big\} - \big\{(x - d_j)^{\ell-1} 1(x  \geq d_j) \big\}\star H_{\delta, k}\Big]\right\|_n \nonumber \\
	&\leq \sum_{j=1}^{J(\epsilon)} |\beta_j| \left\|\big\{(x - d_j)^{\ell-1} 1(x  \geq d_j) \big\} - \big\{(x - d_j)^{\ell-1} 1(x  \geq d_j) \big\}\star H_{\delta, k} \right\|_n.\label{2norm}
\end{align}
Since $\sum_{j=1}^{J(\epsilon)} |\beta_j| = P_{\ell}(f_{\epsilon})/(\ell-1)!$ and $P_{\ell}(f_{\epsilon}) \leq P_{\ell}(f^*)$ by Lemma \ref{lemma:penalty-approx}, to get a rate in $\delta$ on $||f_{\epsilon} - (f_{\epsilon} \star H_{\delta,k})||_n$, we can focus on bounding the remaining $||\cdot||_n$  term above by some function of $\delta$.

Define a new variable $y = x - d_j$, for simplicity of notation. Then we have that the difference between the two functions is given by, letting $g(y) = y^{\ell-1} 1(y \geq 0)$,
\begin{align*}
	g(y) - (g \star H_{\delta,k})(y) &\equiv y^{\ell-1} 1(y \geq 0) - \int_{-\delta}^{\delta} (y-t)^{\ell-1} 1(y \geq t) H_{\delta,k}(t) dt \\
	&= \begin{cases}
		0 \mbox{ if } y < -\delta \\
		- \int_{-\delta}^{\delta} (y-t)^{\ell-1} 1(y \geq t) H_{\delta,k}(t) dt \mbox{ if } -\delta \leq y < 0 \\
		y^{\ell-1} - \int_{-\delta}^{\delta} (y-t)^{\ell-1} 1(y \geq t) H_{\delta,k}(t) dt \mbox{ if } 0 \leq y \leq \delta \\
		0 \mbox{ if } y > +\delta
	\end{cases}
\end{align*}
Therefore, we only have a difference in function values when $|y| \leq \delta$. We focus on each case separately. First, we have, when $0 > y \geq -\delta$:
\begin{align*}
	|g(y) - (g \star H_{\delta,k})(y)| &= \left|\int_{-\delta}^{\delta} (y - t)^{\ell-1} 1(y \geq t) H_{\delta, k}(t) dt\right| \\
	&\leq \int_{-\delta}^{\delta} |(y - t)|^{\ell-1} H_{\delta,k}(t) dt.
\end{align*}
Similarly, when $0 \leq y \leq \delta$:
\begin{align*}
	|g(y) - (g \star H_{\delta,k})(y)| &= \left|y^{\ell-1} - \int_{-\delta}^{\delta} (y-t)^{\ell-1} 1(y \geq t) H_{\delta,k}(t) dt\right| \\
	\text{(Lemma \ref{poly-conv})} &= \left|\int_{-\delta}^{\delta} (y-t)^{\ell-1} \left[1-1(y \geq t)\right] H_{\delta,k}(t) dt\right| \\
	&\leq \int_{-\delta}^{\delta}|y-t|^{\ell-1} H_{\delta,k}(t) dt.
\end{align*}
For each of these cases (i.e., $|y| \leq \delta$), we have
\begin{align*}
	\sup_{|y|\leq \delta}|g(y) - (g \star H_{\delta,k})(y)| &\leq \sup_{|y|\leq \delta} \int_{-\delta}^{\delta}|y-t|^{\ell-1} H_{\delta,k}(t) dt \\
	\text{(letting $s = t/\delta$)} &= \sup_{|y|\leq \delta}\int_{-1}^{1} |y - s\delta|^{\ell-1} H_k(s) ds \\
	&\leq 2^{\ell-1}\delta^{\ell-1} \int_{-1}^{1} H_k(s) ds = 2^{\ell-1}\delta^{\ell-1}.
\end{align*}
Therefore, for knot $d_j$,
\begin{align*}
	||g(y_i) - (g \star H_{\delta,k})(y_i)||^2_n &= \frac{1}{n} \sum_{i=1}^{n} [g(y_i) - (g \star H_{\delta,k})(y_i)]^2 \\
	&\leq \frac{1}{n} \sum_{i=1}^{n} 1(|y_i| \leq \delta) 2^{2(\ell-1)}\delta^{2(\ell-1)} \\
	&\lessapprox P(|x - d_j| \leq \delta) 2^{2(\ell-1)}\delta^{2(\ell-1)} \\
	&= 2^{2(\ell-1)}\delta^{2(\ell-1) + 1} = 2^{2(\ell-1)}\delta^{2\ell-1}
\end{align*}
where the extra $\delta$ comes from the $x_i$ being equally spaced on the unit cube. Putting this together with \eqref{2norm}, we have that, for any $\epsilon > 0$, $||f^* - f^O_{\delta,k}||_n = O_p\{2C'C'' \epsilon P_{\ell}(f^*) + 2^{(\ell-1)}\delta^{(2\ell-1)/2}P_{\ell}(f^*)/(\ell-1)!\}$. Setting $\epsilon = \delta^{(2\ell-1)/2}$ completes the proof, with $||f^* - f^O_{\delta,k}||^2_n = O_p(\delta^{2\ell-1})$.
\end{proof}
\begin{manuallemma}{2}[Upper Bound on Penalty Function]\label{penalty}
	Let $f^* \in \mathcal{F}_{\ell}$, and $k > \ell$. Then $P_k(f_{\delta,k}^O) = O(1/\delta^{k-\ell})$. Specifically,
	\begin{align*}
		P_k(f_{\delta,k}^O) \leq 2 P_{\ell}(f^*)C_{k,k-\ell} /  \delta^{k-\ell}
	\end{align*}
	where $C_{k,k-\ell}$ is the constant defined in Proposition \ref{kernels}.
\end{manuallemma}
\begin{proof}
Since $f^{O}_{k,\delta}$ is defined as the convolution of $f^* \in \mathcal{F}_{\ell}$ and $H_{\delta,k}$, then we can directly apply Lemma \ref{lemma:tv-convolution} to say that $P_k(f_{\delta,k}^O) = O\{1/\delta^{(k-\ell)}\}$.
\end{proof}

\begin{manualtheorem}{2}[Improved Convergence Rates]\label{main-thm}
	Let $f^* \in \mathcal{F}_{\ell}$. If we choose $$\lambda_n = n^{-2k(k + \ell - 1)/(4k\ell - 1)} \{CP_{\ell}(f^*)\}^{(k^{-1} - 2)/(k^{-1} + 2)}$$ with $C$ from Lemma 2, and a $k$th order total variation penalty is used to estimate $\hat{f}$ in (3), then $\delta(n) = O\left\{n^{-2k/(4k\ell - 1)}\right\}$ balances the two terms in \eqref{inequality}, resulting in
	\begin{align*}
		||f^* - \hat{f}||_n^2 = O_p\left\{n^{\frac{-2k (2\ell - 1)}{4k\ell - 1}} \right\}.
	\end{align*}
\end{manualtheorem}
\begin{proof}
	Let $f_{\delta,k}^O$ be defined as in Lemmas \ref{approx-error} and \ref{penalty}. By Lemma \ref{penalty}, for each $\delta$, $f_{\delta,k}^O  \in \mathcal{F}_{k}$, since it has a finite $k$th order total variation. Then applying Theorem 1 and the results of Lemmas \ref{approx-error} and \ref{penalty}, we get that
	\begin{align}\label{eq:theorem2}
		||f^* - \hat{f}||_n^2 \leq O_p(\delta^{2\ell-1}) + O_p\{\delta^{-2(k-\ell)/(2k+1)}n^{-2k/(2k+1)}\}.
	\end{align}
	Optimizing over $\delta$ as a function of $n$, we find that $\delta = n^{-2k/(4lk - 1)}$ balances the approximation error and penalty term. Specifically, we set $\delta = n^{c}$ and solve for $c$ to get equality of the exponents:
    \begin{align*}
        \c(2\ell-1) &= -\frac{c 2(k-\ell)}{(2k+1)}-\frac{2k}{2k+1} \\
        c(2\ell-1)(2k+1) + c 2(k-\ell) &= -2k \\
        c &= \frac{-2k}{(2\ell-1)(2k+1) +  2(k-\ell)} \\
        c &= \frac{-2k}{4\ell k - 1}.
    \end{align*} 
Thus, $\delta = n^{c}$ with $c$ given above balances the two sources of error. Plugging this $\delta$ into \eqref{eq:theorem2}, we get the desired $O_p\left\{n^{-2k(2\ell-1)/(4k\ell - 1)} \right\}$ rate.
\end{proof}

\section{Simulation Results from \citet{simonConvergenceRatesNonparametric2021}}\label{app:sim}

To determine whether their theoretical convergence rates were sharp, \citet{simonConvergenceRatesNonparametric2021} performed simulations of estimating $f^*$ with misspecified penalties, using the data generating process described in Section 1 of the main text. Specifics for estimating $\hat{f}$ are available in \citet{simonConvergenceRatesNonparametric2021}. The experiments had varying sample sizes $n$, which allowed them to determine MSE as a function of $n$, providing estimates of the convergence rates. In Table \ref{tab-1}, we include their results, adding two additional columns: their theoretical rate and our theoretical rate, for each scenario. These are obtained by plugging in values of $k$ and $\ell$ for the specific $f^*$ and penalty into the theoretical rates given in Table 1 of our main text. We see that our rates match their rates, which are very close to the empirical rates, when $f^* \in \mathcal{F}_1$. For $\mathcal{F}_2$, we see our theoretical rate (-0.783) is much closer to the empirical rate (-0.780) than their theoretical rate (-0.750).

\begin{table}[htbp]
\centering
\caption{Comparison of empirical results with the theoretical convergence rates derived in this paper, and in \citet{simonConvergenceRatesNonparametric2021}. $f^*$ is the true data-generating function.}\label{tab-1}
\begin{tabular}{c|ccc}
\hline
\multirow{2}{*}{Penalty} & \multirow{2}{*}{Empirical Rate (SE)} & Theoretical Rate in & \multirow{2}{*}{Our Theoretical Rate} \\
& & \citet{simonConvergenceRatesNonparametric2021} & \\
\hline
\hline
    & \multicolumn{3}{c}{$f^*(x) = 3 I(x > 0.5) \in \mathcal{F}_{1}$} \\
  \hline
  \hline
  $P_2$ & -0.572 (0.006) & -0.571 & -0.571 \\
  $P_3$ & -0.548 (0.005) & -0.546 & -0.546 \\
  \hline
  \hline
  & \multicolumn{3}{c}{$f^*(x) = 3 (x - 0.5)_+ \in \mathcal{F}_2$} \\
  \hline
  \hline
  $P_3$ & -0.780 (0.018) & -0.750 & -0.783 \\
  \hline
\end{tabular}
\end{table}

\newpage
\bibliographystyle{apalike}
\bibliography{bib}